\documentclass[12pt,psfig,reqno]{amsart}
\usepackage{mathrsfs}
\usepackage[colorlinks,
            linkcolor=cyan,
            anchorcolor=cyan,
            citecolor=cyan
            ]{hyperref}
\usepackage{txfonts}
\usepackage{amscd}
\usepackage{cite}
\usepackage{epsfig}
\usepackage{verbatim}
\usepackage{mathdots}
\usepackage{amssymb}
\usepackage{amsfonts}
\usepackage{amsbsy}
 \usepackage{graphicx}
 \usepackage{bm}
 \usepackage{epstopdf}
 \usepackage{latexsym,amsmath,amssymb,amsthm,amsfonts,arydshln,enumerate,graphicx}
 \usepackage[usenames]{color} 
 \usepackage{cancel}
 \numberwithin{equation}{section}
 
\newtheorem{prop}{Proposition}[section]
\newtheorem{lem}[prop]{Lemma}

\newtheorem{defi}{Definition}[section]

\newtheorem{thm}[prop]{Theorem}

\newtheorem{exam}[prop]{Example}

\begin{document}
\begin{sloppypar}
\baselineskip=17pt
\title{On the spectrality of a class of Moran measures}
\author{yali zheng}
\address{yali zheng, School of Mathematics, Hunan University, Changsha, 410082, P.R. China}
\email{zheng92542021@163.com}
\author{Yingqing Xiao$^*$}
\address{Yingqing Xiao, School of Mathematics and Hunan Province Key Lab of Intelligent Information Processing and Applied Mathematics, Hunan University, Changsha, 410082, P.R. China}
\email{ouxyq@hnu.edu.cn}

\date{\today}
\thanks{The research is supported in part by National Natural Science Foundation of China (Grant Nos. 12071118t).}
\thanks{$^{\mathbf{*}}$Corresponding author}
\begin{abstract}
  In this paper, we study the spectrality of a class of Moran measures $\mu_{\mathcal{P},\mathcal{D}}$ on $\mathbb{R}$ generated by $\{(p_n,\mathcal{D}_n)\}_{n=1}^{\infty}$, where $\mathcal{P}=\{p_n\}_{n=1}^{\infty}$ is a sequence of positive integers with $p_n>1$ and $\mathcal{D}=\{\mathcal{D}_{n}\}_{n=1}^{\infty}$ is a sequence of digit sets of $\mathbb{N}$ with the cardinality $\#\mathcal{D}_{n}\in \{2,3,N_{n}\}$. We find a countable set $\Lambda\subset\mathbb{R}$ such that the set $\{e^{-2\pi i \lambda x}|\lambda\in\Lambda\}$ is a orthonormal basis of $L^{2}(\mu_{\mathcal{P},\mathcal{D}})$ under some conditions. As an application, we show that when $\mu_{\mathcal{P},\mathcal{D}}$ is absolutely continuous, $\mu_{\mathcal{P},\mathcal{D}}$ not only is a spectral measure, but also its support set tiles $\mathbb{R}$ with $\mathbb{Z}$.
\end{abstract}

\maketitle
\section{\bf Introduction\label{sect.1}}
\medskip
Given a Borel probability measure $\mu$ with compact support in $\mathbb{R}^{n}$, if there exists a countable set $\Lambda\subset \mathbb{R}^n$ such that the set of exponential functions $E_{\Lambda}:=\{e^{2\pi i\left\langle\lambda,x\right\rangle}|\lambda\in\Lambda\}$ forms an orthonormal basis for ths space $L^{2}(\mu)$ of the square $\mu$-integrable functions, we called the measure $\mu$ a \emph{spectral measure} and the set $\Lambda$ a \emph{spectrum} of the measure $\mu$. When $\mu=\frac{1}{\left|\Omega\right|}dx$, where $\Omega$ is a Borel
set of $\mathbb{R}^n$ with positive Lebesgue measure and $dx$ is the Lebesgue measure, the existence of a spectrum is closely related to the well-known Fuglede conjecture which asserts that $\mu$ is a spectral measure if and only if $\Omega$ is a tile by translation. Here we recall that a set $\Omega$ is called a \emph{tile} by translation if there is a set $T\subset \mathbb{R}^d$ such
that $\Sigma_{t\in T}1_{\Omega}(\lambda-t)=1$ for the Lebesgue measure almost every $\lambda\in \mathbb{R}^d$. The conjecture
is false on $\mathbb{R}^d$ when $d\geq 3$, but it is still open on $\mathbb{R}$ and $\mathbb{R}^2$. We refer the readers to \cite{Fakras2006,Kolountzakis20061,Kolountzakis2006,T}.

There are other probability measures that are not the restriction of the Lebesgue measure on bounded sets, however, they are spectral measures.
%
The first example of a singular, non-atomic, spectral measure was constructed by Jorgensen and Pedersen in \cite{JP}. This discovery has drawn great attention in fractal geometry and Fourier analysis. Since then, lots of techniques were developed to characterize the spectrality of measures, such as operator algebras and Hadamard matrix, see \cite{D1,D2,DHL,DJ,FHW,SR}. Particularly, various new phenomena different from spectral theory of the Lebesgue measure were discovered. For instance, the Fourier series corresponding to different spectral could have completely different convergence property \cite{DHS2,S2}.

With the development of studies, many researchers concentrated their work on investigating the spectrality of all kinds of fractal measures, such as the following Moran measures.

\begin{defi}
Let $\mathcal{P}=\{p_n\}_{n=1}^{\infty}$ be a sequence of expansive matrices and $\mathcal{D}=\{\mathcal{D}_n\}_{n=1}^{\infty}$ be a sequence of digit sets in $\mathbb{R}^{n}$. A Borel probability measure $\mu_{\mathcal{P},\mathcal{D}}$ defined by the infinite convolution of uniformly discrete probability measures
	\begin{equation*}
	\mu_{\mathcal{P},\mathcal{D}}=\delta_{p_{1}^{-1}\mathcal{D}_{1}}*\delta_{(p_{1}p_{2})^{-1}\mathcal{D}_{2}}*\delta_{(p_{1}p_{2}p_{3})^{-1}\mathcal{D}_{3}}*\cdots,
	\end{equation*}
is called a Moran measure,
where $\delta_{E}=\frac{1}{\#E}\sum_{e\in E}\delta_{e}$ is the uniformly discrete measure on $E$ and the infinite convolution converges in a weak sense.
\end{defi}


Moran measures appear frequently and play an important role in dynamic system, fractal geometry and geometric measure theory \cite{FK,H}.
An interesting problem is to determine which Moran measure is a spectral measure. In \cite{SR}, Strichartz first considered the spectrality of Moran measures and obtained the first Moran spectral measure. Afterwards the research about the spectrality of Moran measures flourished. For example, many authors considered the consecutive digit set case, that is $\mathcal{D}_n=\{0,1,\ldots,N_{n}-1\}$ for every $n\in \mathbb{N}$ in a series of papers \cite{AH,DL,AM}. In particular,
An et al. settled completely the spectrality problem of $\mu_{\mathcal{P},\mathcal{D}}$ and showed that $\mu_{\mathcal{P},\mathcal{D}}$ is a spectral measure if and only if $N_{n}|p_{n}$ in \cite{AM}. For the non-consecutive digit set case,
in \cite{HH}, He et al. studied the case $\mathcal{D}_{n}=\{0,d_{n}\}$
with $0<d_{n}<p_{n}$ for all $n\geq 1$ and showed that if $2\mid\frac{p_{n}}{\gcd(d_{n},p_{n})}$ for all $n\geq 1$, then $\mu_{\mathcal{P},\mathcal{D}}$ is a spectral measure.
Later, the case that $\mathcal{D}_{n}=\{0,a_{n},b_{n}\}$ for all $n\in \mathbb{N}$ has also been studied and some sufficient and necessary conditions for $\mu_{\mathcal{P},\mathcal{D}}$ to be a spectral measure have been obtained in a series of papers \cite{FW,D,AHH,LDZ}.
Recently, Shi first considered the mixed case in \cite{S}, that is the cardinality of the digit set $\mathcal{D}_n$ is not a fixed constant. More precisely, Shi studied a new class of Cantor-Moran measure $\mu_{\mathcal{P},\mathcal{D}}$, where $\mathcal{D}_{n}$ is a sequence of integers subsets with the cardinality $\#\mathcal{D}_{n}\in\{2,3\}$ for $n\in \mathbb{N}$ and $\mathcal{P}=\{p_{n}\}_{n=1}^{\infty}$ is a sequence of integers. He constructed a spectrum for $\mu_{\mathcal{P},\mathcal{D}}$.
Inspired by the works of Shi, it is natural to ask the following question:

\textbf{Question:} For a sequence of integers $\mathcal{P}=\{p_{n}\}_{n=1}^{\infty}$ and a sequence of integers subsets $\mathcal{D}=\{\mathcal{D}_{n}\}_{n=1}^{\infty}$ with the cardinality $\#\mathcal{D}_n\in\{2,3,N_n\}$, 
what is the sufficient condition for $\mu_{\mathcal{P},\mathcal{D}}$ to be a spectral measure? Here $\#\mathcal{D}_n=N_n$ means that $\mathcal{D}_{n}=\{0,1,\cdots,N_n-1\}$.

 This paper is devoted to study the spectrality of the Moran measures $\mu_{\mathcal{P},\mathcal{D}}$ and give an answer to the above question. In order to state our result, we need to show some details. It is well-known that if $\{(p_n,\mathcal{D}_{n})\}_{n=1}^{\infty}$ satisfy the condition: $\sup\limits_{n\geq 1}\Big\{\sup\limits_{d\in\mathcal{D}_n}\frac{\left|d\right|}{p_n}\Big\}<\infty$, then the Moran measure $\mu_{\mathcal{P},\mathcal{D}}$ exists and can be written as:
\begin{equation}\label{1.1}
	\mu_{\mathcal{P},\mathcal{D}}=\delta_{p_{1}^{-1} \mathcal{D}_{1}} * \delta_{(p_{1}p_{2})^{-1} \mathcal{D}_{2}} * \cdots * \delta_{(p_{1}p_{2}\cdots p_{n})^{-1} \mathcal{D}_{n}}*\cdots=\mu_{n}*\mu_{>n},
\end{equation}
where $\mu_n$ is the product of the first $n$ terms and $\mu_{>n}$ is the remaining part.
Let $\Omega:=\cup_{n=1}^{\infty}\cup_{i=1}^{3}\Omega^{(i)}_{n}$ and $\mathcal{D}=\{\mathcal{D}_{n}\}_{n=1}^{\infty}\subseteq \Omega$, where $\Omega^{(i)}_{n}$ is defined in the following $\textbf{(T1)-(T3)}$.
\begin{enumerate}
\item[\textbf{(T1)}]$\Omega^{(1)}_{n}=\{0,1,\ldots, N_{n}-1\}$, $N_n$ is a positive integer bigger than $3$ and $N_{n}|p_{n},p_n>N_n$ for all $n\geq 1$.

\item[\textbf{(T2)}] $\Omega^{(2)}_{n}=\{0,a_{n},b_{n}\}$, $0<a_n<b_n$, $\gcd\{a_{n},b_{n}\}=1$, $\{a_{n},b_{n}\}\equiv\{\pm 1\}\pmod 3$, $3|p_{n}$ and $\sup_{n\geq 1}\frac{b_{n}}{p_{n}}<\frac{2}{3}$.

\item[\textbf{(T3)}] $\Omega^{(3)}_{n}=\{0,d_{n}\}$, $0<d_{n}<p_{n}$, $d_{n}=2^{l_{n}}d_{n}^{'}, l_{n}\in\mathbb{N}$ and $d_{n}^{'}$ is a positive odd number. Moreover, $ 2\mid\frac{p_{n}}{\gcd{(d_{n},p_{n})}}$.
\end{enumerate}

In order to study the spectrality of $\mu_{\mathcal{P},\mathcal{D}}$ that satisfies the above conditions $\textbf{(T1)-(T3)}$, we need the following condition, which plays an important role in the study of the spectrality.
\begin{defi}
Let $\mathcal{P}\in M_{n}(\mathbb{Z})$ be a $n\times n$ expansive matrix with integer entries. Let $\mathcal{D},L\subset\mathbb{Z}^{n}$ be finite sets of integer vectors with $N:=\#\mathcal{D}=\#L$. We say that the system $(\mathcal{P},\mathcal{D},L)$ forms a Hadamard triple(or $(\mathcal{P}^{-1}\mathcal{D},L)$ forms a compatible pair or $(\mathcal{P},\mathcal{D})$ is admissible) if the matrix
\begin{equation*}
H=\frac{1}{\sqrt{N}}\left[e^{2\pi i<\mathcal{P}^{-1}d,l>}\right]_{d\in \mathcal{D},l\in L}
\end{equation*}
is unitary, i.e.,$H^{*}H=I$.
\end{defi}
From the Hadamard triple $\{(p_n,\mathcal{D}_n,L_n)\}$ assumption, it is easy to conclude that all Dirac measures $\delta_{{p^{-1}_{n}}\mathcal{D}_n}$ are actually spectral measures.
 When all $p_n$ and $\mathcal{D}_n$ are fixed, \L aba and Wang has shown completely that all such self-similar measures are spectral measures in  \cite{Lw}. Dutkay et al. generalized the result to self-affine measures in $\mathbb{R}^{n}$ in \cite{DHL}. However, such problem become much harder when the measure is a Moran measure. Under various different conditions, some Moran measures generated by Hadamard triples have been proven to be spectral measures in \cite{AFL,LMW2,LLZ2}.
From the definition, it is easy to conclude that
the conditions $\textbf{(T1)-(T3)}$ imply that there exist infinitely many sets $L_n\subset\mathbb{Z}$ such that $(p_n,\mathcal{D}_n,L_n)$ are Hadamard triples (see Lemma \ref{pro3.2}) with no restriction $\sup_{n\geq 1}\{d_n:d_n\in\mathcal{D}_n\}<\infty$ and $\sup_{n\geq 1}\#\{\mathcal{D}_n\}<\infty$ ($N_n$ is unbounded in $\textbf{(T1)}$). 

For the sake of convenience, we introduce some notations from symbolic dynamical system. Set $\mathbb{D}^{n}=\{\alpha_{1}\alpha_{2}\cdots\alpha_{n}:\alpha_{i}\in\mathcal{D}^{\ast}_{i},1\leq i\leq n \}$, where $\mathcal{D}^{\ast}_{i}=\Big\{-\lfloor\frac{N_{i}}{2}\rfloor,-\lfloor\frac{N_{i}}{2}\rfloor+1,\ldots,N_i-1-\lfloor\frac{N_i}{2}\rfloor\Big\}$, $\lfloor x\rfloor$ denotes the largest integer not
	exceeding $x$. For a given $M\in \mathbb{N}$, we set $\mathbb{N}_{>M}:=\{n\in\mathbb{N}:n>M\}$ and put $\Phi(n):=\#\mathcal{D}_{n}, P_{n}=p_{1}p_{2}\ldots p_{n}$ for all $n\in\mathbb{N}$. According to $\textbf{(T1)-(T3)}$, we obtain $\Phi(n)\mid p_{n}$.

	 For $\sigma=\sigma_{1}\sigma_{2}\sigma_{3}\cdots\in\{-1,1\}^{\mathbb{N}}$and $\alpha=\alpha_{1}\alpha_{2}\cdots\alpha_{n}\in\mathbb{D}^{n}$, we construct the following countable set $\Lambda^{\sigma,\alpha}_{n}$ in terms of $(\{p_{n}\},\{l_{n}\},\sigma,\alpha)$:
	\begin{small}
			\begin{equation}\label{1.8}
		\Lambda_{n}^{\sigma,\alpha}:=\sum_{\substack{\Phi(i)=2\\i\leq n}}P_{i}\,\bigg\{0,\frac{\sigma_{i}}{2^{1+l_{i}}}\bigg\}+\sum_{\substack{\Phi(j)=3\\j\leq n}}P_{j}\,\bigg\{0,\frac{1}{3},-\frac{1}{3}\bigg\}+\sum_{\substack{\Phi(m)=N_{m}\\m\leq n}}P_{m}\,\frac{\alpha_{m}}{N_{m}},
		\end{equation}
	\end{small}
and let
	\begin{equation}
		\Lambda^{\sigma,\alpha}=\bigcup_{n=1}^{\infty}\Lambda_{n}^{\sigma,\alpha}.
	\end{equation}

 The main result of this paper is the following:
	
\begin{thm}\label{thm 1.4}
Let $\mu_{\mathcal{P},\mathcal{D}}$ be a Moran measure defined by \eqref{1.1}. If the pair $\{\mathcal{P},\mathcal{D}\}$ satisfies $\textbf{(T1)-(T3)}$, then $\mu_{\mathcal{P},\mathcal{D}}$ is a spectral measure. Moreover, if $\#\{n:\Phi(n)\geq 3,n\geq 1\}=\infty$, $\Lambda^{\sigma,\alpha}$ is a spectrum of $\mu_{\mathcal{P},\mathcal{D}}$ for any $\sigma\in\{-1,1\}^{{N}^{*}}$.
\end{thm}



An outline of this paper is as follows. In Section $2$, we introduce some basic concepts and notations, give orthonormal sets of $\mu_{\mathcal{P},\mathcal{D}}$ and establish several crucial lemmas. The spectrality properties of $\mu_{\mathcal{P},\mathcal{D}}$ are discussed in Section $3$. Finally, some applications and examples related to our main result are given in Section $4$.

\section{\bf Preliminaries \label{sect.2}}

 The purpose of this section is to collect necessary facts needed in the following sections. For a probability measure $\mu$  with compact support on $\mathbb{R}$, the Fourier transform of $\mu$ is defined by
 \begin{equation}\label{2.111}
 \hat{\mu}(\xi)=\int e^{-2\pi i \left\langle x, \xi\right\rangle}d\mu(x),
 \end{equation}
 where $\langle x, y\rangle$ denote the Euclidean inner product of two points $x,y\in \mathbb{R}^n$.
 Denote the zeros set of $\hat{\mu}$ by $\mathcal{Z}\left(\hat{\mu}\right)$, that is $\mathcal{Z}(\hat{\mu}):=\{\xi\in \mathbb{R}^d:\hat{\mu}\left(\xi\right)=0\}$. We say that $\Lambda$ is an orthogonal set of $\mu$ if $E_{\Lambda}$ is an orthonormal family for $L^{2}(\mu)$.  It is easy to show that $\Lambda$ is an orthogonal set of $\mu$ if and only if $\hat{\mu}\left(\lambda-\lambda^{\prime}\right)=0$ for any $\lambda\neq\lambda^{\prime}\in\Lambda$, which is equivalent to $\left(\Lambda-\Lambda\right)\setminus\{0\}\subset\mathcal{Z}\left(\hat{\mu}\right)$.

 For $\mu_{\mathcal{P},\mathcal{D}}$ defined by \eqref{1.1}, it follows from that
 	\begin{equation*}
 	\hat{\mu}_{\mathcal{P},\mathcal{D}}\left(\xi\right)=\prod_{i \atop\Phi(i)=2}m_{\mathcal{D}_{i}}\left(P_{i}^{-1}\xi\right)\prod_{j\atop\Phi(j)=3}m_{\mathcal{D}_{j}}\left(P_{j}^{-1}\xi\right)\prod_{m \atop\Phi(m)=N_{m}}m_{\mathcal{D}_{m}}\left(P_{m}^{-1}\xi\right),
 	\end{equation*}
 	where $m_{\mathcal{D}_{n}}(\xi)=\frac{1}{\#\mathcal{D}_{n}}\sum_{d_{n}\in\mathcal{D}_{n}}e^{2\pi i\left\langle d_{n},\xi\right\rangle}$ is the mask polynomial of $\mathcal{D}_{n}$. Then a direct calculation gives
 	\begin{equation}\label{eq1.7}
 	\mathcal{Z}\left(\hat{\mu}_{\mathcal{P},\mathcal{D}}\right)=\left(\bigcup_{i \atop\Phi(i)=2} P_{i}\,\frac{2\mathbb{Z}+1}{2d_{i}}\right)\bigcup\left(\bigcup_{j \atop\Phi(j)=2}P_{j}\,\frac{3\mathbb{Z}+\{1,2\}}{3}\right)\bigcup\left(\bigcup_{m \atop\Phi(m)=N_m}P_{m}\,\frac{\mathbb{Z}\setminus N_{m}\mathbb{Z}}{N_{m}}\right).
 	\end{equation}

 Let $Q_{\Lambda}\left(\xi\right)=\sum_{\lambda \in \Lambda}\left|\hat{\mu}\left(\xi+\lambda\right)\right|^{2}$. The following lemma is often used to check whether $\Lambda$ is a spectrum of the measure $\mu$.
 \begin{lem}[\cite{JP}]\label{lem 3.2}
 Let $\mu$ be a Borel probability measure with compact support on $\mathbb{R}^{n}$, and let $\Lambda\subset\mathbb{R}^{n}$ be a countable subset. Then
     \begin{enumerate}
     \item[$\mathrm{(i)}$]  $\Lambda$ is an orthonormal set of $\mu$ if and only if $Q_{\Lambda}\left(\xi\right)\geq1$ for $\xi\in\mathbb{R}^{n}$;
   	\item[$\mathrm{(ii)}$] $\Lambda$  is a spectrum of $\mu$ if and only if $Q_{\Lambda}\left(\xi\right)\equiv 1$ for $\xi\in\mathbb{R}^{n}$;
   	\item[$\mathrm{(iii)}$] If $\Lambda$ is an orthonormal set of $\mu$, then $Q_{\Lambda}\left(\xi\right)$ is an entire function.
     \end{enumerate}	
 \end{lem}

The following lemma indicates that it is reasonable to assume that $d_n\in\mathcal{D}_n\setminus\{0\}$ and $p_n>0$ for all $n\geq 1$.

 \begin{lem}
 Given integer sequence $\{p_n\}_{n=1}^{\infty}$ and integer digit set $\{\mathcal{D}_n\}_{n=1}^{\infty}\subseteq\Omega$ defined as above. Let $\theta_n\in\{-1,1\}$ and $\gamma_n\in\{0,\max\left|\mathcal{D}_n\right|\}$, then $\mu_{\{p_n\},\{\mathcal{D}_n\}}$ is a spectral measure if and only if $\mu_{\{\theta_np_n\},\{\mathcal{D}_n+\gamma_n\}}$ is a spectral one.
 \end{lem}

 \begin{proof}
 For all $n\geq 1$, let
 \begin{equation*}
\theta_np_n=\left\{\begin{array}{l}p_n, p_n>0, \\ -p_n, p_n<0,\end{array}\right.
  \end{equation*}
 and
 \begin{equation*}
 \mathcal{D}_n+\gamma_n=\left\{\begin{array}{l}\mathcal{D}_n, \;\text{all}\; d_n\in\mathcal{D}_n\setminus\{0\} \;\text{are positive},\\ \mathcal{D}_n+\max\left|\mathcal{D}_n\right|, \;\text{not all}\; d_n\in\mathcal{D}_n\setminus\{0\}\;\text{are positive}. \end{array}\right.
 \end{equation*}
 For any $\xi\in\mathbb{R}$, it follows from \eqref{2.111} that
 \begin{equation*}
 \hat{\mu}_{\{\theta_np_n\},\{\mathcal{D}_n+\gamma_n\}}\left(\xi\right)=\prod_{n=1}^{\infty}\frac{1}{\#\mathcal{D}_n}\sum_{d_n\in\mathcal{D}_n}e^{\frac{2\pi i(d_n-\gamma_n)\xi}{\theta_n p_n}}=\prod_{n=1}^{\infty}\frac{1}{\#\mathcal{D}_n}e^{\frac{-2\pi i \gamma_n\xi}{\theta_n p_n}}\sum_{d_n\in\mathcal{D}_n}e^{\frac{2\pi id_n\xi}{\theta_n p_n}}.
 \end{equation*}
Hence, $ \left|\hat{\mu}_{\{\theta_np_n\},\{\mathcal{D}_n+\gamma_n\}}\left(\xi\right)\right|=\left|\hat{\mu}_{\{p_n\},\{\mathcal{D}_n\}}\left(\xi\right)\right|$ and the conclusion follows directly from Lemma \ref{lem 2.1} $(ii)$.
 \end{proof}

 Next we introduce some properties of compatible pairs which will be used in the sequel.

 \begin{lem}[\cite{DHL}]\label{pro 1.3}
 Let $\mathcal{P}$ be an $n\times n$ expanding matrix and $\mathcal{D},L\subseteq\mathbb{Z}^{n}$ be two finite subsets of $\mathbb{Z}^{n}$ with the same cardinality. Then the following statements hold:
 \begin{enumerate}
 \item [$\mathrm{(i)}$] $\left(\mathcal{P}^{-1}\mathcal{D},L\right)$ forms a compatible pair.
 \item [$\mathrm{(ii)}$] $\delta_{\mathcal{P}^{-1}\mathcal{D}}$ is a spectral measure with spectrum $L$.
 \item [$\mathrm{(iii)}$] $\hat{\delta}_{\mathcal{P}^{-1}\mathcal{D}}\left(l-l^{'}\right)=0$ for distinct $l,l^{'}\in L$.
 \end{enumerate}
 \end{lem}

 Here we give a lemma to reveal the connection between Hadamard triple and the conditions $\textbf{(T1)-(T3)}$.

 \begin{lem}\label{pro3.2}
Let the pair $\{\mathcal{P},\mathcal{D}\}$ satisfy $\textbf{(T1)-(T3)}$, then there exist a sequence of integer subsets $\{L_n\}_{n=1}^{\infty}\subseteq\mathbb{Z}$ such that $\{(p_n,\mathcal{D}_n,L_n)\}$ are Hadamard triples.
 \end{lem}

 \begin{proof}
When the pair $\{\mathcal{P},\mathcal{D}\}$ satisfy $\textbf{(T1)-(T3)}$, from \cite[Proposition 1.2]{AH}, \cite[Lemma 4.2]{LDZ} and \cite[Proposition 2.2]{WX}, it is easy to obtain that there exist a sequence of integer subsets $\{L_n\}_{n=1}^{\infty}\subseteq\mathbb{Z}$ such that $\{(p_n,\mathcal{D}_n,L_n)\}$ are Hadamard triples, so we omit the details.
\end{proof}

In the following, we proceed to give a lemma to compute the Fourier transforms of certain singular measures and yield a series of direct estimates.

\begin{lem}\label{lem2.44}
Let $\{\mathcal{D}_n\}_{n=1}^{\infty}\subseteq\Omega$ and $\{P_n\}_{n=1}^{\infty}$ be defined as above. For any $x\in\mathbb{R}$, the following statements hold:
\begin{enumerate}
 \item [$\mathrm{(i)}$] If $\Phi(n)=N_n$, $\left|\hat{\delta}_{{P^{-1}_{n}}\mathcal{D}_n}\left(x\right)\right|\geq 1-\frac{x^{2}}{6}$.

 \item [$\mathrm{(ii)}$] If $\Phi(n)=2$, $\left|\hat{\delta}_{P_{n}^{-1}\{0,d_{n}\}}(x)\right|\geq 1-\frac{1}{2}\left(\pi \frac{d_{n}x}{P_{n}}\right)^{2}$.

 \item [$\mathrm{(iii)}$] If $\Phi(n)=3$, $\left|\hat{\delta}_{{P_n}^{-1}\{0,a_{n},b_{n}\}}(x)\right|\geq\frac{1}{3}\left|\cos\left(\pi \frac{\left(a_{n}+b_{n}\right)x}{P_{n}}\right)+2\cos\left(\pi \frac{\left(a_{n}-b_{n}\right)x}{P_{n}}\right)\right|$.
\end{enumerate}
\end{lem}

\begin{proof}
Firstly, when $\Phi(n)=N_n$, a simple calculation gives that
 \begin{equation*}
 \left|\hat{\delta}_{{P^{-1}_{n}}\{0,1, \ldots, N_{n}-1\}}\left(x\right)\right|=\left|m_{P_{n}^{-1}\mathcal{D}_{n}}\left(x\right)\right|
 =\left|\frac{\sin (N_{n}\pi P_{n}^{-1}x)}{N_{n}\sin\left(\pi P_{n}^{-1}x\right)}\right|.
 \end{equation*}
 Using the facts that $\left|\sin x\right|\leq \left|x\right|$ and $\left|\frac{\sin x}{x}\right|\geq 1-\frac{x^{2}}{6}$ for $x\in\mathbb{R}\setminus\{0\}$, we have
 \begin{equation}\label{2.3}
 \left|\hat{\delta}_{{P^{-1}_{n}}\{0,1, \ldots, N_{n}-1\}}(x)\right|=\left|\frac{\sin \left(N_{n}\pi P_{n}^{-1}x\right)}{N_{n}\pi P_{n}^{-1}x}\right|\left|\frac{\pi P_{n}^{-1}x}{\sin\left(\pi P^{-1}_{n}x\right)}\right|
 \geq 1-\frac{1}{6}\left(\frac{N_{n}\pi x}{P_{n}}\right)^{2}.
 \end{equation}
 In addition, $\left|\hat{\delta}_{{P^{-1}_{n}}\{0,1, \ldots, N_{n}-1\}}(0)\right|=1$, i.e., \eqref{2.3} still holds at $x=0$.

 Secondly, for the Fourier transform of measure $\delta_{P_{n}^{-1}\{0,d_n\}}$, we have
  \begin{equation*}
  \hat{\delta}_{P_{n}^{-1}\{0,d_{n}\}}(x)=\frac{1}{2}\left(1+e^{-2 \pi i \frac{d_{n}x}{P_{n}}}\right)=e^{-\pi i \frac{d_{n}x}{P_{n}}}\cos\left(\pi \frac{d_{n}x}{P_{n}}\right).
  \end{equation*}
  Then it follows from the inequality $\cos(x)\geq 1-\frac{1}{2}x^{2}$ that
  \begin{equation*}\label{2.1}
  \left|\hat{\delta}_{P_{n}^{-1}\{0,d_{n}\}}(x)\right|=\left|\cos\left(\pi \frac{d_{n}x}{P_{n}}\right)\right|\geq 1-\frac{1}{2}\left(\pi \frac{d_{n}x}{P_{n}}\right)^{2}.
  \end{equation*}

 Finally, when $\Phi(n)=3$, a straightforward deduction yields the following conclusion:
  \begin{equation*}\label{2.3.1}
  \begin{aligned}
  \left|\hat{\delta}_{{P_n}^{-1}\{0,a_{n},b_{n}\}}(x)\right|
  =&\frac{1}{3}\left|e^{\pi i\frac{\left(a_{n}+b_{n}\right)x}{P_{n}}}+2\cos\left(\pi \frac{\left(a_{n}-b_{n}\right)x}{P_{n}}\right)\right|\\
  \geq &\frac{1}{3}\left|\cos\left(\pi \frac{\left(a_{n}+b_{n}\right)x}{P_{n}}\right)+2\cos\left(\pi \frac{\left(a_{n}-b_{n}\right)x}{P_{n}}\right)\right|.
  \end{aligned}
  \end{equation*}
  On the other hand, we give another form of $\left|\hat{\delta}_{P_{n}^{-1}\{0,a_{n},b_n\}}(x)\right|$, that is
  \begin{equation}\label{2.2}
  \begin{aligned}
  \left|\hat{\delta}_{P_{n}^{-1}\{0,a_{n},b_n\}}(x)\right|^{2}=&\frac{1}{9}\left|e^{\pi i\frac{\left(a_{n}+b_{n}\right)x}{P_{n}}}+2\cos\left(-\pi\frac{\left(a_{n}-b_{n}\right)x}{P_{n}}\right)\right|^{2}\\
  =&\frac{1}{9}\left|1+8\cos\left(\pi\frac{a_{n}x}{P_{n}}\right)\cos\left(\pi\frac{b_{n}x}{P_{n}}\right)\cos \left(\pi\frac{\left(a_{n}-b_{n}\right)x}{P_{n}}\right)\right|.
  \end{aligned}
  \end{equation}
 \end{proof}
In order to provide a more accurate estimation for \eqref{2.2}, we are in this position to give a lemma to reveal the minimal points of the two-variable function $f(x,y)=\cos x\cos y\cos(x-y)$ on $\mathbb{R}^{2}$. Although in \cite{S}, Shi has shown some properties of this function, here we give a more precise explanation. It is easy to see that $f(x,y)$ has periods $(2\pi,0)$ and $(0,2\pi)$. It follows that it has global minimal points in the area $\Delta=\{(x,y):-\pi<x\leq\pi,-\pi<y\leq\pi\}$.

 \begin{lem}\label{lem 2.1}

 Let $f(x,y)=\cos x\cos y\cos(x-y)$. For any sequences $\{x_{n}\}_{n=1}^{\infty}$ and $\{y_{n}\}_{n=1}^{\infty}$, if
 \begin{enumerate}
 \item[$\mathrm{(i)}$] $x_{n} y_{n}>0$,$\left|x_n\right|<\frac{2}{3},\left|y_n\right|<\frac{2}{3}$,
 \item[$\mathrm{(ii)}$] $\varlimsup\limits_{n\rightarrow\infty} \left|x_{n}\right|<\frac{2\pi}{3}~\text{and}~\varlimsup\limits_{n\rightarrow\infty} \left|y_{n}\right|<\frac{2\pi}{3}$,
 \end{enumerate}
 then $\inf\limits_{n\geq 1}\{1+8f(x_{n},y_{n})\}>0$.
 \end{lem}

 \begin{proof}
 By a simple calculation, we have
 \begin{equation*}
 \partial_{x}f(x,y)=-\cos(y)\sin(2x-y), \partial_{y}f(x,y)=\cos(x)\sin(x-2y).
 \end{equation*}

 Let $\partial_{x}f(x,y)=0,\partial_{y}f(x,y)=0$, we obtain the zeros of $\partial_{x}f(x,y), \partial_{y}f(x,y)$ satisfying the following system of equations:

 \

 (1) $\left\{\begin{array}{l}y= \pm \frac{\pi}{2} \\ x= \pm \frac{\pi}{2}\end{array}\right.$
 \quad\quad\quad\quad\quad\quad\quad\quad\quad\quad\quad      (2) $\left\{\begin{array}{l}y= \pm \frac{\pi}{2} \\ x-2 y=0, \pm \pi, \pm 2 \pi\end{array}\right.$

(3) $\left\{\begin{array}{l}2 x-y=0, \pm \pi, \pm 2 \pi \\ x= \pm \frac{\pi}{2}\end{array}\right.$\quad\quad\quad\quad\quad\quad
 (4) $\left\{\begin{array}{l}2 x-y=0, \pm \pi, \pm 2 \pi \\ x-2 y=0, \pm \pi, \pm 2 \pi\end{array}\right.$.
 \

 \

 Next, we consider the second-order partial derivatives of $f$ and obtain that

 \begin{equation}\label{eq 2.5}
  \begin{aligned}
  & \partial_{x x} f(x, y)=-2 \cos (y) \cos (2 x-y), \\
  & \partial_{x y} f(x, y)=\cos (2 x-2 y), \\
  & \partial_{y y} f(x, y)=-2 \cos (x) \cos (x-2 y) .
  \end{aligned}
  \end{equation}

  Substituting the solutions of the system of equations $(1),(2),(3),(4)$  into \eqref{eq 2.5} and combining the discriminatory conditions for the minimal value of a binary function, we find that there are only $8$ minimal value points of $f(x,y)$ in the area $\Delta$:
 \begin{equation*}
 \begin{aligned}
 \left(-\frac{2\pi}{3},\frac{2\pi}{3}\right),\; \left(-\frac{\pi}{3},\frac{\pi}{3}\right),\;
 \left(\frac{\pi}{3},-\frac{\pi}{3}\right),\;\left(\frac{2\pi}{3},-\frac{2\pi}{3}\right),\\
 \left(\frac{2\pi}{3},\frac{\pi}{3}\right),\;\left(\frac{\pi}{3},\frac{2\pi}{3}\right),\;
 \left(-\frac{2\pi}{3},-\frac{\pi}{3}\right),\;\left(-\frac{\pi}{3},-\frac{2\pi}{3}\right).
 \end{aligned}
 \end{equation*}

 A regular calculation implies that $f$ has the global minimal value $-\frac{1}{8}$. Hence, taking the distribution of the minimal value points of $f$ and the assumptions together, we obtain that $\inf\limits_{n\geq 1}\{1+8f(x_{n},y_{n})\}>0$.
 \end{proof}

 The following proposition plays a very important role in the subsequent proof.
  \begin{prop}\label{pro 2.4}
   Let $\sigma\in\{-1,1\}^{\mathbb{N}}$, $k\in\mathbb{N}$ and $n\geq k+1$. Suppose $\{\mathcal{P},\mathcal{D}\}$ satisfies $\textbf{(T1)-(T3)}$, the following statements hold for all $\left|\xi\right|\leq1$ and $\lambda\in\Lambda_{k}^{\sigma,\alpha}$.
   \begin{enumerate}
    \item[$\mathrm{(i)}$] If $\Phi(n)=N_{n}$, $\frac{|N_{n}|}{P_{n}}\left|\xi+\lambda\right|\leq\frac{1}{2}h(k,n)$,
    \item[$\mathrm{(ii)}$] if $\Phi(n)=2$, $\frac{d_{n}}{P_{n}}\left|\xi+\lambda\right|<h(k,n)$,
    \item[$\mathrm{(iii)}$] if $\Phi(n)=3$,  $\frac{\left|a_{n}\pm b_{n}\right|}{P_{n}}\left|\xi+\lambda\right|<\frac{4}{3}h(k,n)$,
   \end{enumerate}
  \end{prop}
 where $h(k,n)=\prod_{i=k+1}^{n-1}\frac{1}{\Phi(i)}\left(\prod_{j=1}^{k}\frac{1}{\Phi(j)}+1\right)$.
 \begin{proof}
  For any $\sigma\in\{-1,1\}^{\mathbb{N}}$, $k\in\mathbb{N}$ and $\lambda\in\Lambda_{k}^{\sigma,\alpha}$, we claim that $\frac{\left|\lambda\right|}{P_{k}}\leq 1$. In fact, from \eqref{1.8}, we know
 \begin{equation*}
 \begin{aligned}
 \frac{\left|\lambda\right|}{P_{k}}\leq& \frac{1}{P_{k}}\left(\sum_{\substack{\Phi(i)=2\\ i\leq k}}P_{i}\,\frac{1}{2^{1+l_{i}}}+\sum_{\substack{\Phi(j)=3\\ j\leq k}}P_{j}\,\frac{1}{3}+\sum_{\substack{\Phi(m)=N_{m}\\ m\leq k}}\frac{1}{2}P_{m}\right)\\
 \leq& \sum_{\substack{\Phi(i)=2\\ i\leq k}}\frac{P_{i}}{P_{k}}\cdot\frac{1}{2}+\sum_{\substack{\Phi(j)=3\\ j\leq k}}\frac{P_{j}}{P_{k}}\cdot\frac{1}{3}+\sum_{\substack{\Phi(m)=N_{m}\\ m\leq k}}\frac{P_{m}}{2P_{k}}\\
 =&\sum_{\substack{\Phi(i)=2\\ i\leq k}}\frac{1}{2p_{i+1}p_{i+2}\cdots p_{k}}+\sum_{\substack{\Phi(j)=3\\ j\leq k}}\frac{1}{3p_{j+1}p_{j+2}\cdots p_{k}}+\sum_{\substack{\Phi(m)=N_{m}\\ m\leq k}}\frac{1}{2p_{m+1}p_{m+2}\cdots p_{k}}\\
 \leq& \sum_{i=1}^{k-1}\frac{1}{2\Phi(i+1)\Phi(i+2)\cdots\Phi(k)}+\frac{1}{2}.
 \end{aligned}
 \end{equation*}
 Set $\beta_{k}(\Phi):=\sum_{i=1}^{k-1} \frac{1}{\Phi(i+1) \cdots \Phi(k)}$ and $\Phi_{\min}=\{\Phi(k)\}_{k=1}^{\infty}$. It is easy to obtain that
 \begin{equation*}
 \beta_{k+1}(\Phi)=\frac{1}{\Phi(k+1)}\left(\beta_{k}(\Phi)+1\right)
 \text{~for all $k \in \mathbb{N}$}.
 \end{equation*}
  Noting that $\Phi_{\min}=2$, we have
 $ \beta_{k}(\Phi) \leq \beta_{k}(\Phi_{\min})\leq 1-\frac{1}{2^{k-1}}\leq 1$, thus the claim follows.

 Now some important facts will be shown. Let $n\geq k+1$, using the fact $\Phi(n)\mid p_{n}$ and the claim, one has
 \begin{equation*}
 \frac{\lambda}{P_{n}}=\frac{1}{p_{n}p_{n-1}\cdots p_{k+1}}\frac{|\lambda|}{P_{k}}\leq \prod_{i=k+1}^{n}\frac{1}{\Phi(i)}.
 \end{equation*}
Suppose $|\xi|<1$. When $\Phi(n)=2$, due to the assumption  $d_{n}<p_{n}$ in $\textbf{(T3)}$, we have
 \begin{equation*}\label{2.6}
 \frac{d_{n}}{P_{n}}\left|\xi+\lambda\right|\leq \frac{d_{n}}{p_{n}}\left(\frac{\left|\xi\right|}{P_{n-1}}+\frac{\left|\lambda\right|}{P_{n-1}}\right)<\prod_{i=k+1}^{n-1}\frac{1}{\Phi(i)}\left(\prod_{j=1}^{k}\frac{1}{\Phi(j)}+1\right).
 \end{equation*}

 Continue the similar procedure to the case $\Phi(n)=3$, notice that $\frac{\left|a_{n}\pm b_{n}\right|}{p_{n}}<\frac{4}{3}$ in $\textbf{(T2)}$, we obtain that
 \begin{equation*}\label{2.7}
 \frac{\left|a_{n}\pm b_{n}\right|}{P_{n}}\left|\xi+\lambda\right|\leq \frac{\left|a_{n}\pm b_{n}\right|}{p_{n}}\left(\frac{\left|\xi\right|}{P_{n-1}}+\frac{\left|\lambda\right|}{P_{n-1}}\right)<\frac{4}{3}\prod_{i=k+1}^{n-1}\frac{1}{\Phi(i)}\left(\prod_{j=1}^{k}\frac{1}{\Phi(j)}+1\right).
 \end{equation*}

 Finally, when $\Phi(n)=N_{n}$, the assumption in $\textbf{(T1)}$ implies $\frac{N_{n}}{p_{n}}\leq\frac{1}{2}$, then
 \begin{equation*}\label{2.8}
 \frac{|N_{n}|}{P_{n}}\left|\xi+\lambda\right|\leq \frac{|N_{n}|}{p_{n}}\left(\frac{|\xi|}{P_{n-1}}+\frac{|\lambda|}{P_{n-1}}\right)\leq\frac{1}{2}\prod_{i=k+1}^{n-1}\frac{1}{\Phi(i)}\left(\prod_{j=1}^{k}\frac{1}{\Phi(j)}+1\right).
 \end{equation*}
  Hence, the conclusions hold.
 \end{proof}

 \section{\bf Proof of Theorem \ref{thm 1.4} \label{sect.3}}
 	
 The main goal of this section is to prove Theorem \ref{thm 1.4}. Firstly, we show that $\Lambda^{\sigma,\alpha}$ is an orthogonal set of $\mu_{\mathcal{P},\mathcal{D}}$.

 	\begin{lem}\label{lem 3.1}
 		Let $\mu_{\mathcal{P},\mathcal{D}}$ be defined by \eqref{1.1}. Suppose the pair $\{\mathcal{P},\mathcal{D}\}$ satisfies $\textbf{(T1)-(T3)}$, then $\Lambda_{n}^{\sigma,\alpha}$ is a spectrum of $\mu_{n}$ and $\Lambda^{\sigma,\alpha}$ is an orthogonal set of $\mu_{\mathcal{P},\mathcal{D}}$ for each $\sigma\in\{-1,1\}^{\mathbb{N}}$ and $\alpha\in\mathbb{D}^{n}$.
 	\end{lem}
 	\begin{proof}
 	Recall that
 	\begin{equation*}
 	\Lambda_{n}^{\sigma,\alpha}:=\sum_{\substack{\Phi(i)=2\\i\leq n}}P_{i}\,\bigg\{0,\frac{\sigma_{i}}{2^{1+l_{i}}}\bigg\}+\sum_{\substack{\Phi(j)=3\\j\leq n}}P_{j}\,\bigg\{0,\frac{1}{3},-\frac{1}{3}\bigg\}+\sum_{\substack{\Phi(m)=N_{m}\\m\leq n}}P_{m}\,\frac{\alpha_{m}}{N_{m}}.
 	\end{equation*}
 		For convenience, we set $\mathcal{F}_{i}=\Big\{0,\frac{\sigma_{i}}{2^{1+l_{i}}}\Big\}$, $\mathcal{G}_{m}=\bigg\{\frac{-\lfloor\frac{N_{m}}{2}\rfloor}{N_{m}},\ldots,-\frac{1}{N_m},0,\frac{1}{N_m},\ldots,\frac{N_m-1-\lfloor\frac{N_m}{2}\rfloor}{N_{m}}\bigg\}$ and $\mathcal{M}_{j}=\Big\{0,\frac{1}{3},-\frac{1}{3}\Big\}$ for $1\leq i,j,m\leq n$. Without loss of generality, we assume that $\Phi(1)=\Phi(2)=\cdots=\Phi(k_{1})=2$, $\Phi(k_{1}+1)=\cdots=\Phi(k_{2})=3$ and $\Phi(k_{2}+1)=N_{k_{2}+1},\cdots,\Phi(n)=N_{n}$ for $k_{1},k_{2}\leq n$. Then, for any $\lambda\neq\lambda^{\prime}\in\Lambda_{n}^{\sigma,\sigma^{\prime}}$, we have
 		\begin{equation*}
 		\lambda=\sum_{i=1}^{k_{1}}P_{i}f_{i}+\sum_{j=k_{1}+1}^{k_{2}}P_{j}m_{j}+\sum_{m=k_{2}}^{n}P_{m}g_{m}
 		\end{equation*}
 and
 		 \begin{equation*}
 		 \lambda^{\prime}=\sum_{i=1}^{k_{1}}P_{i}f^{\prime}_{i}+\sum_{j=k_{1}+1}^{k_{2}}P_{j}m^{\prime}_{j}+\sum_{m=k_{2}}^{n}P_{m}g^{\prime}_{m},
 		 \end{equation*}
 	    where $f_{i},f^{\prime}_{i}\in\mathcal{F}_{i},  m_{j}, m^{\prime}_{j}\in\mathcal{M}_{j}, g_{m}, g^{\prime}_{m}\in\mathcal{G}_{m}$. Let $k$ be the first index such that $f_{k}\neq f^{\prime}_{k}$, then
 	    \begin{equation*} \lambda-\lambda^{\prime}=P_{k}\left( f_{k}-f^{\prime}_{k}+\sum_{i=k+1}^{k_{1}}P_{i}P^{-1}_{k}\left(f_{i}-f^{\prime}_{i}\right)+\sum_{j=k_{1}+1}^{k_{2}}P_{j}P^{-1}_{k}\left(m_{j}-m^{\prime}_{j}\right)+\sum_{m=k_{2}+1}^{n}P_{m}P^{-1}_{k}\left(g_{j}-g^{\prime}_{j}\right)\right).
 	   \end{equation*}
 	   It is easy to see that $f_{k}, f^{\prime}_{k}$ and $f_{k}-f^{\prime}_{k}$ belong to $\mathcal{Z}\left(m_{\mathcal{D}_{k}}\right)\left(\mathcal{D}_{k}=\{0,d_{k}\}\right)$. Noting that $2\mid \frac{p_{n}}{\gcd{(d_{n},p_{n})}}$ and $\Phi(n)\mid p_{n}$, we have $$\left(\sum_{i=k+1}^{k_{1}}P_{i}P^{-1}_{k}\left(f_{i}-f^{\prime}_{i}\right)+\sum_{j=k_{1}+1}^{k_{2}}P_{j}P^{-1}_{k}\left(m_{j}-m^{\prime}_{j}\right)+\sum_{m=k_{2}+1}^{n}P_{m}P^{-1}_{k}\left(g_{j}-g^{\prime}_{j}\right)\right)\in\mathbb{Z}.$$ Therefore, $\lambda-\lambda^{\prime}\in P_{k}\left(\mathcal{Z}\left(m_{\mathcal{D}_{k}}\right)+\mathbb{Z}\right)=P_{k}\left(\mathcal{Z}\left(m_{\mathcal{D}_{k}}\right)\right)=\mathcal{Z}\left(\hat{\delta}_{P^{-1}_{k}\mathcal{D}_{k}}\right)\subset\mathcal{Z}\left(\hat{\mu}_{n}\right)$.
 		The above result implies that $\Lambda_{n}^{\sigma,\alpha}$ is an orthogonal set of $\mu_{n}$. Since the dimension of $L^{2}(\mu_{n})$ is 
 		equivalent to the cardinality of the set $\Lambda_{n}^{\sigma,\alpha}$, we conclude that $\Lambda_{n}^{\sigma,\alpha}$ is a spectrum of $\mu_{n}$.

 		Next, we prove that $\Lambda^{\sigma,\alpha}$ is an orthogonal set of $\mu_{\mathcal{P},\mathcal{D}}$. Fix $\sigma\in\{-1,1\}^{\mathbb{N}}$ and $\alpha\in\mathbb{D}_{n}$. Since $0\in\Lambda^{\sigma,\alpha}_{n}$ for all $n$, we have the following inclusion relation
 \begin{equation*}	\Lambda^{\sigma,\alpha}_{1}\subset\Lambda^{\sigma,\alpha}_{2}\subset\Lambda^{\sigma,\alpha}_{3}\ldots.
 \end{equation*}
For $\lambda\neq \lambda^{\prime}\in\Lambda^{\sigma,\alpha}$, there exists an integer $n$ such that $\lambda,\lambda^{\prime}\in\Lambda^{\sigma,\alpha}_{n}$. As $\Lambda^{\sigma,\alpha}_{n}$ is a spectrum of $\mu_{n}$, we have $\hat{\mu}_{n}\left(\lambda-\lambda^{\prime}\right)=0$. According to the fact that $\mu_{\mathcal{P},\mathcal{D}}=\mu_{n}*\mu_{>n}$, we obtain
 		\begin{equation*}
 		\hat{\mu}_{\mathcal{P},\mathcal{D}}\left(\lambda-\lambda^{\prime}\right)=\hat{\mu}_{n}\left(\lambda-\lambda^{\prime}\right)\cdot\hat{\mu}_{>n}\left(\lambda-\lambda^{\prime}\right)=0.
 		\end{equation*}
 		Thus $\Lambda^{\sigma,\alpha}$ is an orthogonal set of $\mu_{\mathcal{P},\mathcal{D}}$.
 	\end{proof}

%
%

 Now we can  obtain the following lemma which gives a sufficient condition for the spectrality of $\mu_{\mathcal{P},\mathcal{D}}$. The proof follows a similar pattern to that of \cite[Theorem 2.3]{AHH} or \cite[Theorem 1.4]{LMW2}, thus we omit it.

 \begin{lem}\label{lem 4.3}
 Let $\sigma\in\{-1,1\}^{\mathbb{N}}$. Suppose that there exists an increasing sequence $\{n_{k}\}_{k=1}^{\infty}\subset\mathbb{N}$ and $\varepsilon>0, \delta>0$ such that $\left|\hat{\mu}_{>n_{k}}\left(\xi+\lambda\right)\right|>\varepsilon$ for all $\left|\xi\right|<\delta$ and $\lambda\in\Lambda^{\sigma,\alpha}_{n_{k}}$. Then $\Lambda^{\sigma,\alpha}$ is a spectrum of $\mu_{\mathcal{P},\mathcal{D}}$.
 \end{lem}

 Noting that
  $\left|\hat{\mu}_{>n_{k}}\right|=\left|\hat{\delta}_{P^{-1}_{n_{k}+1}\mathcal{D}_{n_{k}+1}}\right|\cdot\left|\hat{\mu}_{>n_{k}+1}\right|$,
  we can obtain a non-zero bound of $\left|\hat{\mu}_{>n_{k}+1}\right|$ under some assumptions, which is what the following result shows. In this manner, it remains to consider the lower bound of $\left|\hat{\delta}_{P^{-1}_{n_{k}+1}\mathcal{D}_{n_{k}+1}}\right|$, it will be discussed in Lemma \ref{lem 5.3} and Lemma \ref{lem 4.6}.


 \begin{prop}\label{lem 4.4}
 Let $\{n_{k}\}_{k=1}^{\infty}\subset\mathbb{N}_{>7}$ be an increasing sequence such that $\Phi(n_{k}+1)\geq 3$.  Suppose the pair $\{\mathcal{P},\mathcal{D}\}$ satisfies $\textbf{(T1)-(T3)}$, then there exists $\epsilon>0$ such that
 \begin{equation*}
 \prod_{j=2}^{\infty}\left|m_{\mathcal{D}_{n_k}+j}\left(\frac{\xi+\lambda}{P_{n_k+j}}\right)\right|>\epsilon,
 \end{equation*}
 for all $|\xi|<1$ and $\lambda\in\Lambda_{n_{k}}^{\sigma,\alpha}$ for all $\sigma\in\{-1,1\}^{N^{*}}$.
 \end{prop}
 \begin{proof}
 By definition,
 \begin{equation*}
 \prod_{j=2}^{\infty}\left|m_{\mathcal{D}_{n_k}+j}(\frac{\cdot}{P_{n_k+j}})\right|=\prod_{\Phi(i)=2\atop i\geq n_{k}+2}\left|\hat{\delta}_{P^{-1}_{i}\{0,d_{i}\}}(\cdot)\right|\cdot \prod_{\Phi(j)=3\atop j\geq n_{k}+2}\left|\hat{\delta}_{P^{-1}_{j}\{0,a_{j},b_{j}\}}(\cdot)\right|\cdot
 \prod_{\Phi(m)=N_{m}\atop m\geq n_{k}+2}\left|\hat{\delta}_{P^{-1}_{m}\{0,1,\ldots,N_{m}-1\}}(\cdot)\right|.
 \end{equation*}

 Let $|\xi|<1$ and $\lambda\in\Lambda_{n_{k}}^{\sigma,\alpha}$ for all $\sigma\in\{-1,1\}^{N^{*}}$.
 Now we estimate the lower bound of $\left|\hat{\delta}_{P^{-1}_{i}\{0,d_{i}\}}\left(\xi+\lambda\right)\right|$, $\left|\hat{\delta}_{P^{-1}_{j}\{0,a_{j},b_{j}\}}\left(\xi+\lambda\right)\right|$ and $\left|\hat{\delta}_{P^{-1}_{m}\{0,1,\ldots,N_{m}-1\}}\left(\xi+\lambda\right)\right|$ respectively for all $\Phi(i)=2$, $\Phi(j)=3$, $\Phi(m)=N_{m}$, $i,j,m\geq n_{k}+2$.

 We first pay attention to $\left|\hat{\delta}_{P^{-1}_{j}\{0,a_{j},b_{j}\}}\left(\xi+\lambda\right)\right|, j\geq n_{k}+2$.
 By Proposition \ref{pro 2.4} $(iii)$ and the assumption $\Phi(n_{k}+1)\geq 3$, clearly,
 \begin{equation}\label{1}
 \pi\frac{\left|a_{j}\pm b_{j}\right|}{P_{j}}\left|\xi+\lambda\right|\leq \frac{4}{3}\cdot\frac{\pi}{3\cdot 2^{j-2-n_{k}}}\left(1+\frac{1}{2^{n_{k}}}\right)
 \leq\frac{4\pi}{9\cdot 2^{j-9}}\left(1+\frac{1}{2^7}\right)\leq \frac{43\pi}{96}.
 \end{equation}

 From Lemma \ref{lem2.44} $(iii)$, we see that
 \begin{equation}\label{2}
 \begin{aligned}
 \left|\hat{\delta}_{P^{-1}_{j}\{0,a_{j},b_{j}\}}\left(\xi+\lambda\right)\right|
 \geq&\frac{1}{3}\left(\cos\left(\pi \frac{\left(a_{j}+b_{j}\right)\left(\xi+\lambda\right)}{P_{j}}\right)+2\cos\left(\pi \frac{\left(a_{j}-b_{j}\right)\left(\xi+\lambda\right)}{P_{j}}\right)\right)\\
 \geq&\frac{1}{3}\left(1-\frac{1}{2}\left(\pi \frac{\left(a_j+b_j\right)\left(\xi+\lambda\right)}{P_j}\right)^{2}+2\left(1-\frac{1}{2}\left(\pi \frac{\left(a_j-b_j\right)\left(\xi+\lambda\right)}{P_j}\right)^{2}\right)\right)\\
 \geq& 1-\frac{1}{2}\left(\pi \frac{\left(a_j\pm b_j\right)\left(\xi+\lambda\right)}{P_j}\right)^{2},
 \end{aligned}
 \end{equation}
 where the first two inequalities come from $\frac{\left|a_{j}\pm b_{j}\right|}{P_{j}}\left|\xi+\lambda\right|<\frac{1}{2}$ and $\cos(x)\geq 1-\frac{1}{2}x^{2}$. Together with \eqref{1} and \eqref{2}, one has
 \begin{equation*}
 \left|\hat{\delta}_{P^{-1}_{j}\{0,a_{j},b_{j}\}}\left(\xi+\lambda\right)\right|\geq 1-\frac{1}{2}\left(\frac{4\pi}{9\cdot 2^{j-9}}\left(1+\frac{1}{2^{7}}\right)\right)^{2}
 \geq 1-\frac{1}{2}\left(\frac{43\pi}{96}\right)^{2}>0.
 \end{equation*}

 Now we deal with $\left|\hat{\delta}_{P^{-1}_{i}\{0,d_{i}\}}\left(\xi+\lambda\right)\right|$, $i\geq n_{k}+2$ and $\Phi(i)=2$.
 Similarly, according to Proposition \ref{pro 2.4} $(ii)$,
 \begin{equation*}
 \pi\frac{d_{i}\left|\xi+\lambda\right|}{P_{i}}
 \leq\frac{\pi}{3\cdot 2^{i-9}}\left(1+\frac{1}{2^{7}}\right)
 <\frac{4\pi}{9\cdot 2^{i-9}}\left(1+\frac{1}{2^7}\right),
 \end{equation*}
 thus from Lemma \ref{lem2.44} $(ii)$, we know that
 \begin{equation*}
 \left|\hat{\delta}_{P_{i}^{-1}\{0,d_{i}\}}\left(\xi+\lambda\right)\right|>1-\frac{1}{2}\left(\frac{4\pi}{9\cdot 2^{i-9}}\left(1+\frac{1}{2^{7}}\right)\right)^{2}>0.
 \end{equation*}

 Finally, continue the similar procedure to  $\left|\hat{\delta}_{P^{-1}_{m}\{0,1,\ldots,N_{m}-1\}}\left(\xi+\lambda\right)\right|$, $m\geq  n_{k}+2$ and $\Phi(m)=N_{m}$.
 It follows from the assumption and Proposition \ref{pro 2.4} $(i)$ that
 \begin{equation}\label{4.9}
 \frac{N_{m}\pi \left|\xi+\lambda\right|}{P_{m}}\leq\frac{1}{2}\cdot\frac{\pi}{3\cdot 2^{m-9}}\left(1+\frac{1}{2^{7}}\right)<\frac{4\pi}{9\cdot 2^{m-9}}\left(1+\frac{1}{2^{7}}\right),
 \end{equation}
 combining \eqref{4.9} and Lemma \ref{lem2.44} $(i)$, we obtain that
 \begin{equation*}
 \left|\delta_{{P^{-1}_{m}}\{0,1, \ldots, N_{m}-1\}}\left(\xi+\lambda\right)\right|
 >1-\frac{1}{2}\left(\frac{4\pi}{9\cdot 2^{m-9}}\left(1+\frac{1}{2^{7}}\right)\right)^{2}>0.
 \end{equation*}
 Based on the above discussion, we conclude that
 \begin{equation*}
 \prod_{j=2}^{\infty}\left|m_{\mathcal{D}_{n_k}+j}\left(\frac{\xi+\lambda}{P_{n_k+j}}\right)\right|\geq C_{0},
 \end{equation*}
 where $C_{0}=\prod\limits_{\Phi(i)=2~\text{or}~3~\text{or}~N_{i}\atop i\geq n_{k}+2}\left(1-\frac{1}{2}\left(\frac{43\pi}{96\cdot 2^{i-9}}\right)^{2}\right)\in(0,1)$ is a constant which is independent of $n_{k}$. The proof is complete.
 \end{proof}

 To reach the desired conclusion, it is required to estimate the lower bound of $\left|m_{\mathcal{D}_{n_k+1}}\left(\frac{\xi+\lambda}{P_{n_k+1}}\right)\right|$.
 In fact, under the assumption $\sup_{n\geq 1}\frac{b_{n}}{p_{n}}<\frac{2}{3}$ in $\textbf{(T1)}$, we first obtain an expected lower bound estimate of $\left|\hat{\delta}_{P_{n_k+1}^{-1}\{0,a_{n_{k}+1},b_{n_{k}+1}\}}\left(\xi+\lambda\right)\right|$.

 \begin{lem}\label{lem 5.3}
Let $0<a_n<b_n$ for all $n\geq 1$. Suppose there exists an increasing subsequence $\{n_{k}\}_{k=1}^{\infty}$ such that $\varlimsup\limits_{k\rightarrow\infty}\frac{b_{n_{k}+1}}{p_{n_{k}+1}}<\frac{2}{3}$, then there exists $\epsilon>0$ and $N\in\mathbb{N}$ such that when $k>N$,

 \begin{equation}
 \left|\hat{\delta}_{P_{n_{k}+1}^{-1}\{0,a_{n_{k}+1},b_{n_{k}+1}\}}\left(\xi+\lambda\right)\right|\geq \epsilon,
 \end{equation}
 for all $|\xi|<1$ and $\lambda\in\Lambda_{n_k}^{\sigma,\alpha}$.
 \end{lem}

 \begin{proof}
 Let $\omega_{1}=\pi\frac{a_{n_{k}+1}\left(\xi+\lambda\right)}{P_{n_{k}+1}}, \omega_{2}=\pi\frac{b_{n_{k}+1}\left(\xi+\lambda\right)}{P_{n_{k}+1}}$. By \eqref{2.2}, we obtain
 \begin{equation*}
 \left|\hat{\delta}_{P_{n_{k}+1}^{-1}\{0,a_{n_{k}+1},b_{n_{k}+1}\}}\left(\xi+\lambda\right)\right|^{2}
 =\left|\frac{1}{9}\left(1+f(\omega_{1},\omega_{2})\right)\right|,
 \end{equation*}
 where $f(\omega_{1},\omega_{2})=\cos\left(\omega_{1}\right)\cos\left(\omega_{2}\right)\cos\left(\omega_{1}-\omega_{2}\right)$. From the assumption, there exist $\varlimsup\limits_{k\rightarrow\infty}\frac{b_{n_{k}+1}}{p_{n_{k}+1}}<l<\frac{2}{3}$ and an integer $K:=K(l)>0$ such that $\frac{b_{n_{k}+1}}{p_{n_{k}+1}}\leq l$ for all $k\geq K$.
 It follows from Proposition \ref{pro 2.4} that
 \begin{equation*}
 \frac{\xi+\lambda}{p_1p_2\cdots p_{n_k}}\in\left(-1-\frac{1}{p_1p_2\cdots p_{n_k}},1+\frac{1}{p_1p_2\cdots p_{n_k}}\right),
 \end{equation*}
 for all $\left|\xi\right|<1$ and $\lambda\in\Lambda_{n_k}^{\sigma,\alpha}$.

 Since $p_n>1$, we have
 \begin{equation*}
 \left|\frac{b_{n_{k}+1}\left(\xi+\lambda\right)}{P_{n_{k}+1}}\right|\leq\left(1+\frac{1}{2^{n_k}}\right)l,\;\;k>K.
 \end{equation*}
 On the other hand, for any $l<l^{'}<\frac{2}{3}$, there exists $K':=K^{'}(l^{'})>K$ such that for any $k>K^{'}$, $\left(1+\frac{1}{2^{n_k}}\right)l<l^{'}$. Hence,
 \begin{equation*}
 \left|\omega_2\right|\leq l^{'}\pi<\frac{2}{3}\pi, \;\;k>K^{'}.
 \end{equation*}
  Due to $0<a_n<b_n$ for all $n\geq 1$, it follows that $\left|\omega_1\right|<\frac{2\pi}{3}$ for any $k>K^{'}$ as well. Denote $\epsilon=\min\bigg\{\frac{1}{9}\left(1+8f\left(\omega_{1},\omega_{2}\right)\right):\left|\omega_{1}\right|\leq l^{'}\pi,\left|\omega_{2}\right|\leq l^{'}\pi\bigg\}$, by lemma \ref{lem 2.1}, it is easy to see $\epsilon>0$ and we have
 \begin{equation*}
 \left|\hat{\delta}_{P_{n_{k}+1}^{-1}\{0,a_{n_{k}+1},b_{n_{k}+1}\}}\left(\xi+\lambda\right)\right|\geq \epsilon^{\frac{1}{2}}, \text{ for any $k>K^{'}$}.
 \end{equation*}
 \end{proof}

 As for the estimation of $\left|\hat{\delta}_{P_{n_{k}+1}^{-1}\{0,1,\ldots,N_{n_{k}+1}-1\}}\left(\xi+\lambda\right)\right|$, the following result is approached easily.

 \begin{lem}\label{lem 4.6}
 For any increasing subsequence $\{n_{k}\}_{k=1}^{\infty}$, when $\mathcal{D}_{n_{k}+1}=N_{n_{k}+1}$ and $\{p_{n_{k}+1},{D}_{n_{k}+1}\}$ satisfies condition $\textbf{(T1)}$ accordingly, then there exists $\epsilon>0$ such that
 \begin{equation*}
 \left|\hat{\delta}_{P_{n_{k}+1}^{-1}\mathcal{D}_{n_{k}+1}}\left(\xi+\lambda\right)\right|\geq \epsilon,
 \end{equation*}
 for all $\left|\xi\right|<1$ and $\lambda\in\Lambda_{n_k}^{\sigma,\alpha}$.
 \end{lem}

 \begin{proof}
 A direct application of Proposition \ref{pro 2.4} gives that
 \begin{equation*}
 \left|\frac{N_{n_{k}+1}\pi(\xi+\lambda)}{P_{n_{k}+1}}\right|=\left|\frac{\pi N_{n_k+1}}{p_{n_k+1}}\cdot\frac{\xi+\lambda}{P_{n_k}}\right|\leq \frac{\pi}{2}\left(\frac{1}{2^{n_k}}+1\right)\leq\frac{3\pi}{4}.
 \end{equation*}
 By Lemma \ref{lem2.44} $(i)$, we have
 \begin{equation*}
 \left|\hat{\delta}_{P_{n_{k}+1}^{-1}\mathcal{D}_{n_{k}+1}}\left(\xi+\lambda\right)\right|\geq 1-\frac{1}{6}\left(\frac{N_{n_{k}+1}\pi \left(\xi+\lambda\right)}{P_{n_{k}+1}}\right)^{2}\geq 1-\frac{1}{6}\left(\frac{3\pi}{4}\right)^{2}:=\epsilon>0.
 \end{equation*}
 The proof is complete.
 \end{proof}

 With full preparations above, we are ready to prove Theorem \ref{thm 1.4}. This issue will be discussed in two cases:
 \begin{enumerate}
 	\item[$\mathrm{(i)}$] $\#\{n:\Phi(n)\geq 3,n\geq 1\}=+\infty$.
 	\item[$\mathrm{(ii)}$] $\#\{n:\Phi(n)\geq 3,n\geq 1\}<+\infty$.
 \end{enumerate}

  When it comes to case $\mathrm{(i)}$, according to Lemma \ref{lem 4.3}, it is crucial to find an increasing subsequence $\{n_{k}\}_{k=1}^{\infty}$ satisfying the conditions of Lemma \ref{lem 4.3}. In fact, combining with Proposition \ref{lem 4.4}, Lemma \ref{lem 5.3} and Lemma \ref{lem 4.6}, it is easy to arrive at the desired conclusion. As for the case $\mathrm{(ii)}$,
 we will construct a new set $\tilde{\Lambda}$ and prove it is a spectrum of $\mu_{\mathcal{P},\mathcal{D}}$ by Lemma \ref{lem 3.2}.

 \textbf{Case 1:} $\#\{n:\Phi(n)\geq 3,n\geq 1\}=+\infty$.

 As an immediate application of Proposition \ref{lem 4.4} and Lemmas \ref{lem 5.3}-\ref{lem 4.6}, we can construct a subsequence $\{n_{k}\}_{k=1}^{\infty}$ such that $\left|\hat{\mu}_{>n_{k}}(\xi+\lambda)\right|$ has a non-zero bound.

 \begin{prop}\label{pro 5.1}
 	Suppose $\mu_{\mathcal{P},\mathcal{D}}$ satisfy the same assumptions with Theorem \ref{thm 1.4}. Suppose $\#\{n:\Phi(n)\geq 3,n\geq 1\}=+\infty$, then $\mu_{\mathcal{P},\mathcal{D}}$ is a spectral measure with spectrum $\Lambda^{\sigma,\alpha}$ for any $\sigma\in\{-1,1\}^{\mathbb{N}}$.
 \end{prop}
 \begin{proof}
 We choose a subsequence $\{n_{k}\}_{k=1}^{\infty}$ satisfying
 	\begin{enumerate}
 		\item $n_{k}\geq 7$,
 		\item $\Phi(n_{k}+1)\geq 3$ for all $k\in\mathbb{N}$.
 	\end{enumerate}
 		Due to the assumption $\#\{n:\Phi(n)\geq 3,n\geq 1\}=+\infty$, the subsequence $\{n_{k}\}_{k=1}^{\infty}$ indeed exists. Obviously, it satisfies the conditions of Proposition \ref{lem 4.4}, so there exists $\epsilon>0$ such that
 			\begin{equation*}
 			\left|\hat{\mu}_{>n_{k}+1}(\xi+\lambda)\right|\geq \epsilon,
 			\end{equation*}
 	for all $\lambda\in\Lambda^{\sigma,\alpha}_{n_{k}}$ and $\left|\xi\right|<1$.
 	On the other hand, it follows from the assumptions and Lemmas \ref{lem 5.3}-\ref{lem 4.6} that there exists $\epsilon^{'}$ such that
 	\begin{equation*}
 	\left|\hat{\delta}_{P_{n_{k}+1}^{-1}\mathcal{D}_{n_{k}+1}}(\xi+\lambda)\right|>\epsilon^{'},
 	\end{equation*}	
 	for all $\lambda\in\Lambda^{\sigma,\alpha}_{n_{k}}$ and $\left|\xi\right|<1$.
 Hence, we obtain
    \begin{equation*}
 	\left|\hat{\mu}_{>n_{k}}(\xi+\lambda)\right|=\left|\hat{\mu}_{>n_{k}+1}(\xi+\lambda)\right|\cdot\left|\hat{\delta}_{P_{n_{k}+1}^{-1}\mathcal{D}_{n_{k}+1}}(\xi+\lambda)\right|\geq \epsilon\epsilon^{'}.
 	\end{equation*}
 Finally, it follows from Lemma \ref{lem 4.3} that the spectrality conclusion holds. The proof is complete.
 \end{proof}

\textbf{Case 2:} ${\#\{n:\Phi(n)\geq 3,n\geq 1\}<+\infty}$.

In this case, we know that there only exist finitely many $\mathcal{D}_{n}$ with $\#\mathcal{D}_{n}\geq 3$ in $\mu_{\mathcal{P},\mathcal{D}}$. So we will divide the measure $\mu_{\mathcal{P},\mathcal{D}}$ into two parts to realize our proof.

 \begin{prop}\label{pro 5.3}
 	Let $\mu_{\mathcal{P},\mathcal{D}}$ satisfy the same assuptioms with Theorem \ref{thm 1.4}. Suppose $\#\{n:\Phi(n)\geq 3,n\geq 1\}<+\infty$, then $\mu_{\mathcal{P},\mathcal{D}}$ is a spectral measure.
 \end{prop}
 \begin{proof}
 Since $\#\{n:\Phi(n)\geq 3,n\geq 1\}<+\infty$, we conclude that there only exist finitely many $\mathcal{D}_{n}$ with $\#\mathcal{D}_{n}\geq 3$ in $\mu_{\mathcal{P},\mathcal{D}}$. So there exists $N>0$ such that $\Phi(n)=2$, i.e. $\mathcal{D}_{n}=\{0,d_{n}\}$ for all $n>N$. In this way, $\mu_{\mathcal{P},\mathcal{D}}$ can be rewritten as
 	\begin{equation*}
 	\mu_{\mathcal{P},\mathcal{D}}(\cdot)=\mu_{N}(\cdot)\ast\nu(p_1p_2\cdots p_N\cdot),
 	\end{equation*}
 	where
 	\begin{equation*}
 	\mu_{N}=\delta_{P^{-1}_{1}\mathcal{D}_1}(\cdot)\ast\delta_{P^{-1}_{2}\mathcal{D}_2}(\cdot)\ast\ldots\delta_{P^{-1}_{N}\mathcal{D}_N}(\cdot),
 	\end{equation*}
 	 and  $$\nu(\cdot)=\delta_{p^{-1}_{N+1}\{0,d_{N+1}\}}(\cdot)\ast\delta_{(p_{N+1}p_{N+2})^{-1}\{0,d_{N+2}\}}(\cdot)\ast\ldots.$$
 Let
 \begin{equation*}
 \Lambda_{N}=\sum_{\substack{\Phi(i)=2\\i\leq N}}P_{i}\,\bigg\{0,\frac{\sigma_{i}}{2^{1+l_{i}}}\bigg\}+\sum_{\substack{\Phi(j)=3\\j\leq N}}P_{j}\,\bigg\{0,\frac{1}{3},-\frac{1}{3}\bigg\}+\sum_{\substack{\Phi(m)=N_{m}\\m\leq N}}P_{m}\,\frac{\alpha_{m}}{N_{m}},
 \end{equation*}
 where $\sigma_{i}\in\{-1,1\}$ and $l_{i},\alpha_{m}$ is defined by \eqref{1.8}. According to Lemma \ref{lem 3.1}, we know that $\Lambda_{N}$ is a spectrum of $\mu_{N}$.
 On the other hand, it follows from Theorem $1.1$ and Theorem $1.3$ of \cite{HH} that $\nu$ is a spectral measure with integer spectrum $\Lambda^{'}$(associated with $\sigma$). Define a new $\tilde{\Lambda}$ by
 \begin{equation*}
 \tilde{\Lambda}=\Lambda_{N}+p_1p_2p_3\ldots p_N\Lambda^{'}.
 \end{equation*}
 Now we are in a position to prove that $\tilde{\Lambda}$ is a spectrum of $\mu_{\mathcal{P},\mathcal{D}}$. By Lemma \ref{lem 3.2}, it is equivalent to show that for all $\xi\in\mathbb{R}$, $Q(\xi)=\sum_{\substack{\tilde{\lambda}\in\tilde{\Lambda}}}\left|\hat{\mu}_{\mathcal{P},\mathcal{D}}(\tilde{\lambda}+\xi)\right|^{2}=1$. In fact,
 \begin{equation*}
 \begin{aligned}
 Q\left(\xi\right)=&\sum_{\substack{\tilde{\lambda}\in\tilde{\Lambda}}}\left|\hat{\mu}_{\mathcal{P},\mathcal{D}}\left(\tilde{\lambda}+\xi\right)\right|^{2}\\
 =&\sum_{\substack{\lambda^{'}\in\Lambda^{'}}}\sum_{\substack{\lambda\in\Lambda_{N}}}\left|\hat{\mu}_{N}\left(\xi+\lambda\right)\right|^{2}\left|\hat{\nu}\left(\frac{\xi+\lambda}{p_1p_2\ldots p_N}+\lambda^{'}\right)\right|^{2}\\
 =&\sum_{\substack{\lambda\in\Lambda_{N}}}\left|\hat{\mu}_{N}\left(\xi+\lambda\right)\right|^{2}\sum_{\substack{\lambda^{'}\in\Lambda^{'}}}\left|\hat{\nu}\left(\frac{\xi+\lambda}{p_1p_2\ldots p_N}+\lambda^{'}\right)\right|^{2}\\
 =&1.
 \end{aligned}
 \end{equation*}
 From the above discussion, the conclusion can be reached that $\mu_{\mathcal{P},\mathcal{D}}$ is a spectral measure with spectrum $\tilde{\Lambda}$. The proof is complete.
 \end{proof}

 \begin{proof}[Proof of Theorem \ref{thm 1.4}]
 The desired result comes directly from Proposition \ref{pro 5.1} and Proposition \ref{pro 5.3}.
 \end{proof}

 \section{\bf Examples and applications}

 In this section, we give some examples and applications related to our main result. The following example indicates that the conditions in $\textbf{(T1)-(T3)}$ are essential in some sense.

 \begin{exam}\label{ex 5.3}
 Let $\mathcal{D}_1=\{0,1,2\}$, $\mathcal{D}_2=\{0,5,6\}$, $\mathcal{D}_n=\{0,3\}$ for $n\geq 3$ and $p_n=2$ for $n\geq 1$. Then $\mu_{\{p_n\},\{\mathcal{D}_n\}}$ is not a spectral measure.
 \end{exam}

 \begin{proof}
 It is easy to see that $\{\mathcal{P},\mathcal{D}\}$ do not meet most of the conditions in $\textbf{(T1)-(T3)}$. A direct calculation shows that
 \begin{equation*}
 \begin{aligned}
 \mu_{\{p_n\},\{\mathcal{D}_n\}}&=\delta_{2^{-1}\{0,1,2\}}\ast\delta_{2^{-2}\{0,5,6\}}\ast\delta_{2^{-3}\{0,3\}}\ast\delta_{2^{-4}\{0,3\}}\ast\ldots\\
 &=\delta_{\{0,\frac{1}{2},1,\frac{5}{4},\frac{3}{2},\frac{7}{4},2,\frac{9}{4},\frac{5}{2}\}}\ast\frac{4}{3}\mathcal{L}_{[0,\frac{3}{4}]}\\
 &=\frac{4}{27}\mathcal{L}_{[0,\frac{1}{2}]\cup[\frac{3}{4},1]\cup[3,\frac{13}{4}]}+\frac{8}{27}\mathcal{L}_{[\frac{1}{2},\frac{3}{4}]\cup[1,\frac{3}{2}]\cup[\frac{11}{4},3]}+\frac{4}{9}\mathcal{L}_{[\frac{3}{2},\frac{11}{4}]},
 \end{aligned}
 \end{equation*}
 where $\mathcal{L}$ denotes the Lebesgue measure in $\mathbb{R}$. It has been showed that an absolutely continuous spectral measure must be uniform on its support \cite{DL1}. It is clear that $\mu_{\{p_n\},\{\mathcal{D}_n\}}$ is not uniformly distributed on its support $[0,\frac{13}{4}]$. Hence, $\mu_{\{p_n\},\{\mathcal{D}_n\}}$ is not a spectral measure.
 \end{proof}

The following examples imply that some conditions in $\textbf{(T1)-(T3)}$ maybe not necessary. To be specific, Example \ref{ex 5.4} shows that the hypothesis $2\mid\frac{p_n}{\gcd\left(d_n,p_n\right)}$ in $\textbf{T(3)}$ may potentially be refined.
 \begin{exam}\label{ex 5.4}
 Let $(\mathcal{D}_1,p_1)=(\{0,1\},4),(\mathcal{D}_2,p_2)=(\{0,2\},9)$ and $(\mathcal{D}_n,p_n)=(\{0,1\},4)$ for $n\geq 3$. Then $\mu_{\{p_n\},\{\mathcal{D}_n\}}$ is a spectral one.
 \end{exam}

 \begin{proof}
  Note that $2\nmid \frac{9}{\gcd(2,9)}$ in $(\mathcal{D}_2,p_2)$ do not satisfy $2\mid \frac{p_2}{\gcd(d_2,p_2)}$ in $\textbf{T(3)}$. The specific proof is similar to \cite[Example 5.3]{HH}, so we omit it.
 \end{proof}

 Furthermore, Example \ref{ex 5.5} not only gives a concrete characterization to our main result but also says that the assumption $\sup_{n\geq 1}\Big\{\frac{b_n}{p_n}\Big\}<\frac{2}{3}$ may be modified.

 \begin{exam}\label{ex 5.5}
 Let $(\mathcal{D}_1,p_1)=(\{0,2\},4),(\mathcal{D}_2,p_2)=(\{0,1,2\},3k),k\in\mathbb{N}$ and $(\mathcal{D}_n,p_n)=(\{0,1,2,\ldots,n\},2(n+1))$ for $n\geq 3$. Then $\mu_{\{p_n\},\{\mathcal{D}_n\}}$ is a spectral one.
 \end{exam}

 \begin{proof}

 Define that
 \begin{equation*}
 \nu=\delta_{12^{-1}\{0,1,2\}}\ast\delta_{(12k)^{-1}(2\cdot 4)^{-1}\{0,1,2,3\}}\ast\delta_{(12k)^{-1}(2\cdot4)^{-1}(2\cdot5)^{-1}\{0,1,2,3,4\}}\ast\ldots,
 \end{equation*}
 then $\mu_{\{p_n\},\{\mathcal{D}_{n}\}}=\delta_{4^{-1}\{0,2\}}\ast\nu$.
 Obviously, it follows from Lemma \ref{pro3.2} that there exists $L=\{0,1\}$ such that $L$ is a spectrum of $\delta_{4^{-1}\{0,2\}}$, and \cite[Theorem 1.4]{AH} implies that $\nu$ is a spectral measure. A simple calculation gives that $\mathcal{Z}\left(\hat{\nu}\right)\subseteq\mathbb{Z}$. Hence $\mu_{\{p_n\},\{\mathcal{D}_n\}}$ is a spectral one by \cite[Theorem 1.5]{HLL}.
 \end{proof}

  Up to now, the spectrality of measure $\mu_{\{p_n\},\{\mathcal{D}_{n}\}}$ satisfying $\textbf{(T1)-(T3)}$ has been proved in Theorem \ref{thm 1.4}. In fact, more is true: inspired by \cite{LLZ}, some kinds of $\mu_{\{p_n\},\{\mathcal{D}_{n}\}}$ are absolutely continuous with respect to Lebesgue measure and their supports are tiles of $\mathbb{R}$, as will be explained next. Though a more general case has been demonstrated in \cite{LLZ}, here we will give a proof for completeness.

  \begin{thm}\label{thm 1.5}
  Let $\mu_{\mathcal{P},\mathcal{D}}$ be defined by \eqref{1.1} with $\Phi(n)=p_n$ for all $n\geq 1$. Suppose the pair $\{\mathcal{P},\mathcal{D}\}$ satisfies $\textbf{(T2)}$ and $\textbf{(T3)}$, then the supports of $\mu_{\mathcal{P},\mathcal{D}}$ tile $\mathbb{R}$ with tiling set $\mathbb{Z}$.
  \end{thm}

 \begin{proof}
 From Theorem \ref{thm 1.4}, we know $\mu_{\mathcal{P},\mathcal{D}}$ is a spectral one with spectrum $\Lambda\subset\mathbb{Z}$. When the pair $\{\mathcal{P},\mathcal{D}\}$ satisfies $\textbf{(T2)}$ and $\textbf{(T3)}$, Lemma \ref{pro3.2} tells us that there exist $\{L_n\}_{n=1}^{\infty}\subseteq\mathbb{Z}$ such that $\{(p_n,\mathcal{D}_n,L_n)\}$ is a sequence of Hadamard triples. The assumption $\Phi(n)=p_n$ for all $n\geq 1$ ensures that $\mu_{\mathcal{P},\mathcal{D}}$ is an absolutely continuous measure   by Theorem $5.3$ and Theorem $5.7$ of \cite{LLZ}. From \eqref{eq1.7}, a simple calculation gives
 \begin{equation*}
 \mathcal{Z}\left(\hat{\mu}_{\mathcal{P},\mathcal{D}}\right)=\left(\underset{i\geq1,\Phi(i)=2}\cup 2^{i}\frac{2\mathbb{Z}+1}{2}\right)\bigcup\left(\underset{j\geq1, \Phi(j)=3}\cup 3^{j}\frac{\mathbb{Z}\setminus3\mathbb{Z}}{3}\right)=\mathbb{Z}\setminus\{0\},
 \end{equation*}
 hence we obtain that $\Lambda=\mathbb{Z}$ and it follows from \cite{DL1} and Theorem 2.2 of \cite{Lw2} that  $\mu_{\mathcal{P},\mathcal{D}}=\chi_{T}dx$, where $\chi_{T}$ is a characteristic function on $T$ (support set of $\mu_{\mathcal{P},\mathcal{D}}$). Finally, by Theorem 2.1 of \cite{Lw2}, we can conclude that $T$ is a fundamental domain of lattice $\mathbb{Z}$, i.e., $T\oplus\mathbb{Z}=\mathbb{R}$ (a fundamental domain of a lattice $L$ is a set $\mathcal{D}$ such that $\underset{l\in L}\cup (\mathcal{D}+l)$ tiles $\mathbb{R}^{n}$ almost everywhere).
 \end{proof}

 Finally, we give an example related to Theorem \ref{thm 1.5}.
 \begin{exam}
 Let $(\mathcal{D}_{2k-1},p_{2k-1})=(\{0,1\},2)$ and $(\mathcal{D}_{2k},p_{2k})=(\{0,1,2\},3)$ for all $k\geq 1$. Then $\mu_{\mathcal{P},\mathcal{D}}$ is a spectral measure and the support of $\mu_{\mathcal{P},\mathcal{D}}$ tiles $\mathbb{R}$ with unique tiling set $\mathbb{Z}$.
 \end{exam}

 \begin{proof}
 It is easy to see that $\mu_{\mathcal{P},\mathcal{D}}$ is a spectral one with spectrum $\mathbb{Z}$. Now we show that the support of $\mu_{\mathcal{P},\mathcal{D}}$ is a tile of $\mathbb{R}$.
 Let
 \begin{equation*}
 \mu=\delta_{2^{-1}\{0,1\}}\ast\delta_{6^{-1}\{0,1,2\}}\ast\delta_{2^{-1}6^{-1}\{0,1\}}\ast\delta_{6^{-2}\{0,1,2\}}\ast\delta_{2^{-1}6^{-2}\{0,1\}}\ast\cdots.
 \end{equation*}

 A simple calculation gives that
 \begin{equation*}
 T(\mathcal{P},\mathcal{D})=\left\{\sum_{n=1}^{\infty}(p_{n}\cdots p_{1})^{-1}d_{n}:d_n
 \in\mathcal{D}_n\right\}=[0,1].
 \end{equation*}
  where $T(\mathcal{P},\mathcal{D})$ is the support set of measure $\mu_{\{p_n\},\{\mathcal{D}_n\}}$. Hence the conclusion holds.
 \end{proof}

\end{sloppypar}

\begin{thebibliography}{100}


	    \bibitem{AH}  L.X. An, X.G. He, A class of spectral Moran measure, {\em J. Funct. Anal.}, {\bf 266} (2014), 343-354.
	
	    \bibitem{AFL}L.X. An, X.Y. Fu, C.K. Lai, On spectral Cantor-Moran measures and a variant of Bourgain's sum of sine problem, {\em Adv. Math.}, {\bf 349} (2019), 84-124.
	
		
		\bibitem{AHH} L.X. An, L. He, X.G. He, Spectrality and non-spectrality of the Riesz product measures with three
		elements in digit sets, {\em J. Funct. Anal.}, {\bf 277} (2019), 255-278.

        \bibitem{AM} L.X. An, Q. Li, M.M. Zhang, Characterization of spectral Cantor-Moran measures with consecutive digits, Preprint.

        \bibitem{D1} X.R. Dai, When does a Bernoulli convolution admit a spectrum?, {\em Adv. Math.}, {\bf 231} (2012), 1681-1693.


        \bibitem{D2} X.R. Dai, Spectra of Cantor measures, {\em Math. Ann.}, {\bf 366} (2016), 1621-1647.

        \bibitem{D} D.X. Ding, Spectral property of certain fractal measures, {\em J. Math. Anal. Appl.}, {\bf 451} (2017), 623-628.

        \bibitem{DHL} D.E. Dutkay, J. Haussermann, C.K. Lai, Hadamard triples generate self-affine spectral measures, {\em Trans. Amer. Math. Soc.}, {\bf 371} (2019), 1439-1481.

        \bibitem{DJ} D.E. Dutkay, P. Jorgensen, Analysis of orthogonality and of orbits in affine iterated function systems, {\em Math. Z.}, {\bf 256} (2007), 801-823.

        \bibitem{DL1} D.E. Dutkay, C.K. Lai, Uniformity of measures with Fourier frames, {\em Adv. Math.}, {\bf 252} (2014), 684-707.


        \bibitem{DL} Q.R. Deng, M.T. Li, Spectrality of Moran-type self-similar measures on $\mathbb{R}$, {\em J. Math. Anal. Appl.}, {\bf 506} (2022), 125547.


        \bibitem{DHS2} D.E. Dutkay, D.G. Han, Q.Y. Sun, Divergence of the mock and scrambled Fourier series on fractal measures, {\em Trans. Amer. Math. Soc.}, {\bf 366}(2014), 2191-2208.

\bibitem{Fakras2006} B. Farkas, M. Matolcsi, P. Mora, On Fuglede's conjecture and the existence of universal spectra, \emph{J. Fourier Anal. Appl.}, {\bf 12}(2006), 483-494.

		\bibitem{FK} K.J. Falconer, Fractal Geometry, Mathematical Foundations and Applications, {\em John Wiley $\&$ Sons, Ltd.}, 2014.


	   \bibitem{FHW} Y.S. Fu, X.G. He, Z.X. Wen, Spectra of Bernoulli convolutions and random convolutions, {\em J. Math. Pures Appl.}, {\bf 116} (2018), 105-131.
		
	   \bibitem{FW} Y.S Fu, Z.X. Wen, Spectrality of infinite convolutions with three-element digit sets, {\em Monatsh. Math.}, {\bf 183} (2017), 465-485.
				


	   \bibitem{H} J. Hutchinson, Fractals and self-similarity, {\em Indiana Univ. Math. J.}, {\bf 30} (1981), 713-747.
		
       \bibitem{HH} L. He, X.G. He, On the Fourier orthonormal bases of Cantor-Moran measures, {\em J. Funct. Anal.}, {\bf 272} (2017), 1980-2004.
	
	
       \bibitem{HLL} X.G. He, C.K. Lai, K.S. Lau,  Exponential spectra in $L^{2}(\mu)$, {\em Appl. Comput. Harmon. Anal.}, {\bf 34} (2013), 327-338.

		\bibitem{JP} P. Jorgenson, S. Pederson, Dense analytic subspaces in fractal $L^{2}$-spaces, {\em J. Anal. Math.}, {\bf 75} (1998), 185-228.
		
\bibitem{Kolountzakis20061}M. Kolountzakis, M. Matolcsi, Complex Hadamard matrices and the spectral set conjecture. \emph{Collect. Math.}, {\bf 57}(2006), 281-291

\bibitem{Kolountzakis2006}M. Kolountzakis, M. Matolcsi, Tiles with no spectra, \emph{Forum Math.}, {\bf18}(2006), 519-528.
			
		
	     \bibitem{LDZ} Z.Y. Lu, X.H. Dong, P.F. Zhang, Spectrality of some one-dimensional Moran measures, {\em J. Fourier Anal. Appl.}, {\bf 28} (2022), 63.
		
		
		\bibitem{LLZ2} J.S. Liu, Z.Y. Lu, T. Z hou, Spectrality of Moran-Sierpinski type measures, {\em J. Funct. Anal.}, {\bf 284} (2023), 109820.
		
		\bibitem{LLZ} J.S. Liu, Z.Y. Lu, T. Zhou, The spectrality of Cantor-Moran measure and Fuglede's Conjecture,  arXiv:2305.12135,2023.
		
	
		\bibitem{LMW2} W.X. Li, J.J. Miao, Z.Q. Wang, Spectrality of random convolutions generated by finitely many Hadamard triples, arXiv:2203.11619,2022.
	
		\bibitem{Lw} I. \L aba, Y. Wang, On spectral Cantor measures, {\em J. Funct. Anal.}, {\bf 193} (2002), 409-420.
		
	    \bibitem{Lw2} J.C. Lagarias, Y. Wang, Spectral sets and factorizations of finite abelian groups, {\em J. Funct. Anal.}, {\bf 145} (1997), 73-98.
		
		\bibitem{SR} R. Strichartz, Mock Fourier series and transforms associated with certain Cantor measures, {\em J. Anal. Math.}, {\bf 81} (2000), 209-238.	
		
		
		\bibitem{S2} R.S. Strichartz, Convergence of mock Fourier series, {\em J. Anal. Math.}, {\bf 99} (2006) 333-353.
		
		
		\bibitem{S} R.X. Shi, Spectrality of a class of Cantor-Moran measures, {\em J. Funct. Anal.}, {\bf 276} (2019), 3767-3794.
			
		
		\bibitem{T} T. Tao, Fuglede's conjecture is false in 5 and higher dimensions, {\em Math. Res. Lett.}, {\bf 11} (2004), 251-258.
		
		\bibitem{WX} S. Wu, Y.Q. Xiao, Spectrality of a class of infinite convolutions on $\mathbb{R}$, \emph{Nonlinearity}, {\bf 37(5)}(2024), 055015.
		
		
\end{thebibliography}
\end{document}